\documentclass[11pt]{amsart}
   
\usepackage{amsthm,amsmath,amssymb,amscd,amsfonts,latexsym}

\theoremstyle{plain}
\newtheorem{theorem}{Theorem}[section]

\newtheorem{corollary}[theorem]{Corollary}

\newtheorem{def-thm}[theorem]{Definition-Theorem}
\newtheorem{lemma}[theorem]{Lemma}

\theoremstyle{definition}

\newtheorem{remark}[theorem]{Remark}

\newcommand{\PP}{\mathbb{P}}

\newcommand{\RR}{\mathbb{R}}

\newcommand{\NN}{\mathbb{N}}

\DeclareMathOperator{\Ric}{Ric}

\DeclareMathOperator{\kod}{kod}

\DeclareMathOperator{\divisor}{div}

\allowdisplaybreaks

\begin{document}

\title{Scalar curvature and uniruledness on projective manifolds}

\begin{abstract} 
It is a basic tenet in complex geometry that {\it negative} curvature corresponds, in a suitable sense, to the absence of rational curves on, say, a complex projective manifold, while {\it positive} curvature corresponds to the abundance of rational curves. In this spirit, we prove in this note that a projective manifold $M$ with a K\"ahler metric with positive total scalar curvature is uniruled, which is equivalent to every point of $M$ being contained in a rational curve. We also prove that if $M$ possesses a K\"ahler metric of total scalar curvature equal to zero, then either $M$ is uniruled or its canonical line bundle is torsion. The proof of the latter theorem is partially based on the observation that if $M$ is not uniruled, then the total scalar curvatures of all K\"ahler metrics on $M$ must have the same sign, which is either zero or negative.
\end{abstract}

\author{Gordon Heier} \author{Bun Wong}

\address{Department of Mathematics\\University of Houston\\4800 Calhoun Road, Houston, TX 77204\\USA}
\email{heier@math.uh.edu}

\address{Department of Mathematics\\UC Riverside\\900 University Avenue\\ Riverside, CA 92521\\USA}
\email{wong@math.ucr.edu}

\thanks{The first author is partially supported by the National Security Agency under Grant Number H98230-12-1-0235.}
\maketitle

\section{Introduction and statement of the results} 
A complex projective manifold $M$ of dimension $n$ is uniruled if and only if there exists a dominant rational map from $\PP^1\times N$ onto $M$, where $N$ is a complex projective variety of dimension $n-1$. Equivalently, $M$ is uniruled if and only if there exists a rational curve passing through every point of $M$. It is a well-known open problem whether all complex projective manifolds of Kodaira dimension $-\infty$ are uniruled. In dimensions one and two, this is classical. In dimension three, this follows from the Minimal Model Program in Algebraic Geometry. The converse of this statement is easily verified in any dimension.\par

From a differential geometric point of view, complex projective manifolds of Kodaira dimension $-\infty$ resemble compact K\"ahler manifolds of positive scalar curvature. Yau \cite{Yau_74} proved that the Kodaira dimension of a compact K\"ahler manifold with positive total scalar curvature must be $-\infty$. Moreover, it is not hard to see that every uniruled manifold is rationally dominated by a K\"ahler manifold with positive scalar curvature (see Section \ref{sec_concl_rem}). Nonetheless, the converse of this statement is difficult to establish without a good method to guarantee the existence of rational curves under the assumption of positivity of scalar curvature.\par

One highlight of the development of this type of problem along differential geometric lines was when Siu-Yau \cite{siu_yau} proved the Frankel Conjecture, which states that a compact K\"ahler manifold of positive bisectional curvature is biholomorphic to a projective space. The crucial step of their proof relied on the production of a rational curve using the analyticity of a stable harmonic map under the positive bisectional curvature assumption. Around the same time, Mori \cite{Mori_79} settled a more general conjecture of Hartshorne, which states that a compact complex manifold with ample tangent bundle is biholomorphic to projective space. In his proof, Mori introduced the {\it bend and break} technique, which is a method, partially relying on arguments in positive characteristic, to show the existence of rational curves based on certain positivity assumptions. Mori's technique has important consequences. For example, let $C$ be a curve passing through a point $x$ on a complex projective manifold $X$, and let $K$ denote the canonical line bundle of $X$. If the intersection number $-K.C$ is positive, then there is a rational curve passing through the same point $x$. In particular, a {\it Fano manifold} $X$, i.e., a compact complex manifold with positive first Chern class $c_1(X)=-c_1(K)$, is uniruled. Note that, due to Yau's proof of Calabi's Conjecture, a compact complex manifold has positive first Chern class if and only if it possesses a K\"ahler metric with positive Ricci curvature. Actually, more is true, namely that Fano manifolds are rationally connected. We refrain from giving detailed attributions of these results and instead refer the reader to the monographs \cite{kollar_book}, \cite{debarre_book} and the references contained therein.\par

The first main result proven in this note is that, instead of assuming positivity of the Ricci curvature, it suffices to make the much weaker assumption of positivity of the total scalar curvature to obtain uniruledness, as differential geometers would desire. Note that on the spectrum of the usual positivity notions in differential geometry, positive total scalar curvature, which is expressed in terms of the positivity of just one number, is the weakest notion of positivity. Our result is the following.\par

\begin{theorem}\label{mthm1}
Let $M$ be a projective manifold with a K\"ahler metric with positive total scalar curvature. Then $M$ is uniruled.
\end{theorem}

The general philosophy is that a complex projective manifold which is not uniruled must satisfy certain generic positivity conditions in terms of its cotangent bundle in light of \cite{MM}, \cite{M1}. Concretely, we will be using the result of Boucksom-Demailly-Paun-Peternell \cite{BDPP_04}, who proved that if a projective manifold  is not uniruled then it has a pseudo-effective canonical line bundle. On the whole, our results should be considered as differential geometric interpretations of the theory of bend and break and the recent work \cite{BDPP_04} on the duality of certain cones (see \cite[Theorem 2.2]{BDPP_04}, reproduced below as Theorem \ref{bdpp_mthm}).\par
In general, it is not true that a K\"ahler manifold with positive holomorphic sectional curvature has positive first Chern class (see the examples on Hirzebruch surfaces constructed by Hitchin in \cite{Hitchin}). It is thus not possible to directly conclude the uniruledness of a K\"ahler manifold with positive holomorphic sectional curvature from the earlier results. However, it does follow from a pointwise argument due to Berger \cite{Berger} that the scalar curvature (and thus also the total scalar curvature) of a K\"ahler metric of positive holomorphic sectional curvature is also positive. Thus, we obtain the following corollary of Theorem \ref{mthm1}.

\begin{corollary}
Let $M$ be a projective manifold with a K\"ahler metric with positive holomorphic sectional curvature. Then $M$ is uniruled. 
\end{corollary}

\begin{remark}
It was proven in \cite{heier_lu_wong} in dimension 3 that a K\"ahler manifold with negative holomorphic sectional curvature has negative first Chern class. This is conjectured to be true even in arbitrary dimension and represents a significant distinction between the cases of positive and negative curvature.
\end{remark}
Our second main result is a characterization of projective manifolds with zero total scalar curvature.
\begin{theorem}\label{mthm2}
Let $M$ be a projective manifold with a K\"ahler metric with total scalar curvature equal to zero. Then either $M$ is uniruled or the canonical line bundle of $M$ is a torsion line bundle.
\end{theorem}
A key ingredient in the proof of Theorem \ref{mthm2} will be the following statement on the uniqueness of the sign of the total scalar curvature in the absence of uniruledness. It appears to be new and interesting in its own right, so we state it as another theorem.
\begin{theorem}\label{no_two_signs}
Let $M$ be a projective manifold which is not uniruled. Then the total scalar curvatures of all K\"ahler metrics on $M$ must have the same sign, which is either zero or negative.
\end{theorem}\par
In this paper, we take a {\it Calabi-Yau manifold} to be a compact complex manifold admitting a Ricci-flat K\"ahler metric. For a projective manifold $X$, using the results in \cite{Yau}, \cite{Kawamata_85}, \cite{Nakayama}, one can easily show that being Calabi-Yau is equivalent to the canonical line bundle of $X$ being a torsion line bundle. Thus, the conclusion of Theorem \ref{mthm2} can be restated as saying that either $M$ is uniruled or Calabi-Yau.\par
We observe that, in the uniruled case, the sign of the (total) scalar curvature may be arbitrary. The simplest example is that of $\PP^1\times N$, where $N$ is a compact hyperbolic Riemann surface and both manifolds come with their standard constant scalar curvature K\"ahler metrics. After multiplying one of the factors in the product metric with an appropriate positive number, the scalar curvature can be made to have an arbitrary sign.\par
We also state the following theorem, which is stronger than Theorem \ref{no_two_signs}. It simply follows from combining the statements of Theorem \ref{mthm2} and \ref{no_two_signs}, but it may be of particular interest to readers focussed on complex differential geometry.
\begin{theorem}\label{CY}
Let $M$ be a projective manifold which is not uniruled. Then either all K\"ahler metrics on $M$ have negative total scalar curvature or $M$ is Calabi-Yau.
\end{theorem}
Since a compact K\"ahler manifold with negative holomorphic sectional curvature cannot be Calabi-Yau due to \cite[Theorem 2.3]{heier_lu_wong} and does not admit any rational curves, the following corollary is immediate from Theorem \ref{CY}.
\begin{corollary}
Let $M$ be a projective manifold with a K\"ahler metric with negative holomorphic sectional curvature. Then the total scalar curvature of any K\"ahler metric on $M$ is negative. 
\end{corollary}
Finally, we would like to point out that it is commonly expected that a projective manifold with pseudo-effective canonical line bundle has non-negative Kodaira dimension (see \cite[p.~1]{BDPP_04}). If a proof of this open problem were known, most of our results would become easy to prove using the method in \cite{Yau_74} relating scalar curvature and the Poincar\'e-Lelong formula (see page \pageref{PL_formula}). Namely, due to \cite{BDPP_04}, non-uniruledness could then be directly translated to the existence of a non-zero pluricanonical section. For a very recent development concerning such problems of abundance conjecture-type, we refer to the manuscript \cite{siu} by Siu.

\section{The case of positive total scalar curvature}
In this section, we prove Theorem \ref{mthm1}. We begin by first recalling some standard notions and notations (see \cite{Dem}, \cite{BDPP_04}, \cite{debarre_book}).\par

A current $T$ of bidegree $(1,1)$ on a complex manifold of dimension $n$ is {\it positive} if for every choice of smooth $(1,0)$ forms $a_1,\ldots,a_{n-1}$, the distribution
$${\sqrt{-1}}^{n-1}T\wedge a_1\wedge \bar a_1  \wedge \ldots \wedge a_{n-1}\wedge \bar a_{n-1}$$
is a positive measure.\par

A line bundle $F$ on a projective manifold $X$ is {\it pseudo-effective} if its first Chern class $c_1(F)$ is in the closed cone in $H^{1,1}(X)$ generated by the classes of effective divisors. For our purposes, we think of pseudo-effective line bundles in terms of an equivalent differential-geometric characterization: A line bundle $F$ on a projective manifold $X$ is {\it pseudo-effective} if it carries a singular hermitian metric $h$, locally given by $h=e^{-\varphi}$ with $\varphi\in L^1_{\text{loc}}$, such that its (globally well-defined) curvature current 
$$ \Theta(h)=\sqrt{-1}\partial\bar\partial \varphi$$ is a positive current of bidegree $(1,1)$. \par

We note also that, under the canonical isomorphism mapping the usual Dolbeault cohomology with coefficients in $F$ to the corresponding cohomology for currents, the class of $\Theta(h)$ corresponds to the class $2\pi c_1(F)$, where $c_1(F)$ is the first Chern class of $F$.\par

Furthermore, for a K\"ahler metric $g=\sum_{i,j=1}^n g_{i\bar{j}} dz_i\otimes d\bar{z}_{j}$, its {\it Ricci curvature form} is given by
\begin{equation*}
\Ric(g)=-\sqrt{-1}\partial\bar\partial \log\det(g_{i\bar{j}}).
\end{equation*}
By a result of Chern, the class of the form $\frac{1}{2\pi}\Ric(g)$ is equal to $c_1(M)=c_1(-K)$, where $K$ is the canonical line bundle of $M$. The {\it scalar curvature} $s$ of $g$ is defined to be the trace of $-\sqrt{-1}\Ric(g)$ with respect to a unitary frame. Finally, the {\it total scalar curvature} of $g$ is defined to be
$$\int_M s\frac{\omega^n}{n!},$$
where $\omega=\frac{\sqrt{-1}}{2}\sum_{i,j=1}^n g_{i\bar{j}} dz_i\wedge d\bar{z}_{j}$ is the {\it K\"ahler form associated to $g$}.\par
To prove Theorem \ref{mthm1}, let $g$ now denote the K\"ahler metric on $M$ with positive total scalar curvature. It follows from linear algebra and the definition of scalar curvature that
$$\Ric (g)\wedge \omega^{n-1} = \frac 2 n s\, \omega^n,$$
where $s$ is the scalar curvature of $g$. \par
Now, assume that $M$ is not uniruled. By \cite[Corollary 0.3]{BDPP_04}, the canonical line bundle $K$ of $M$ is pseudo-effective, i.e., it carries a singular hermitian metric $h$ whose curvature current $\Theta(h)$ is a positive current of bidegree $(1,1)$. Since both $\Theta(h)$ and $-\Ric (g)$ represent $2\pi c_1(K)$, we have
\begin{align}
~&\ \int_M \Theta(h)\wedge \omega^{n-1}\label{1}\\
=&\ -\int _M \Ric (g)\wedge \omega^{n-1}\nonumber\\
=&\ -\int_M  \frac 2 n s\, \omega^n\nonumber\\
<&\ 0,\nonumber
\end{align}
where the last inequality is due to the positivity of the total scalar curvature.
On the other hand, the expression in \eqref{1} is non-negative, which yields a contradiction. The reason for the non-negativity is the positivity of $\Theta(h)$ and the fact that $\omega\wedge\ldots\wedge \omega$ can be written as a global sum of the form $\sum_{k=1}^{N} {\sqrt{-1}}^{n-1}a^{(k)}_1\wedge \bar a^{(k)}_1  \wedge \ldots \wedge a^{(k)}_{n-1}\wedge \bar a^{(k)}_{n-1},$ where the $a^{(k)}_i$, $k=1,\ldots,N$, $i=1,\ldots,n-1$, are globally defined smooth $(1,0)$ forms. To justify this, we notice that every hermitian form can locally be represented in terms of a {\it coframe} of smooth $(1,0)$ forms $\varphi_1,\ldots,\varphi_n$ as a sum $\sqrt{-1}\sum_{k=1}^n \varphi_k\wedge \bar \varphi_k$. The globalization is then achieved by means of a partition of unity. For complete details, we refer the reader to the proofs of Lemma \ref{lin_alg_claim} and Lemma \ref{int_pos}, which contain identical arguments.

\section{Uniqueness of the curvature sign in the absence of uniruledness}
To prepare for the proof of Theorem \ref{mthm2} in the next section, we now prove Theorem \ref{no_two_signs}. According to Theorem \ref{mthm1}, positive total scalar curvature implies uniruledness. Thus, in order to derive a contradiction, we may assume that there are two K\"ahler metrics on $M$ such that one has zero total scalar curvature and one has negative total scalar curvature. We denote these K\"ahler metrics and their K\"ahler forms with $g_0,\omega_0$ and $g_-,\omega_-$, respectively.\par
Since $M$ is compact, we can choose a sufficiently small positive number $\varepsilon$ such that $\omega=\omega_0-\varepsilon\omega_-$ is still a K\"ahler form. Also, let $n=\dim M$, let
\begin{equation*}
A=\omega^{n-1}_0,
\end{equation*}
and
\begin{equation*}
B=(\varepsilon\omega_-)^{n-1}=\varepsilon^{n-1}\omega_-^{n-1}.
\end{equation*}
\begin{lemma}\label{lin_alg_claim} There exists a finite open covering $(U_\nu)$ of $M$ and positive integers $N_\nu$ such that for every $\nu$ and $k_\nu\in \{1,\ldots,N_\nu\}$, there exist smooth $(1,0)$ forms $a^{(\nu,k_\nu)}_1,\ldots,a^{(\nu,k_\nu)}_{n-1}$ on $U_\nu$ such that, on $U_\nu$, we have
$$A-B={\sqrt{-1}}^{n-1} \sum_{k_\nu=1}^{N_\nu} a^{(\nu,k_\nu)}_1\wedge \bar a^{(\nu,k_\nu)}_1  \wedge \ldots \wedge a^{(\nu,k_\nu)}_{n-1}\wedge \bar a^{(\nu,k_\nu)}_{n-1}.$$
\end{lemma}
\begin{proof} In the case $n=2$, the statement of the lemma is true simply because $A-B=\omega_0-\varepsilon \omega_-$ is a hermitian form. Namely, it is well-known that on a small open set, there exists a {\it coframe} of smooth $(1,0)$ forms $\varphi_1,\ldots,\varphi_n$ such that 
\begin{equation}\label{coframe}
A-B=\sqrt{-1}\sum_{k=1}^n \varphi_k\wedge \bar \varphi_k.
\end{equation}
This proves the lemma in the case $n=2$. We now assume $n\geq 3$.\par
For real numbers $a,b$, it is a well-known fact that
$$a^{n-1}-b^{n-1}=(a-b)(a^{n-2}+a^{n-3}b+\ldots + b^{n-2}).$$
Because $\omega_0,\omega_-$ are each $(1,1)$-forms, we have
$$\omega_0\wedge\omega_-=\omega_-\wedge\omega_0,$$
and the same identity thus holds in our situation:
\begin{align*} 
A-B & = \omega_0^{n-1}- (\varepsilon\omega_-)^{n-1}\\
&=(\omega_0-\varepsilon \omega_-)\wedge(\omega_0^{n-2}+(\omega_0^{n-3})\wedge(\varepsilon\omega_-)+\ldots+(\varepsilon\omega_-)^{n-2}).
\end{align*}
Now note that $\omega_0-\varepsilon \omega_-, \omega_0,\omega_-$ locally all have representations in terms of a coframe as on the right hand side of \eqref{coframe}. Substituting these representations into the second line of the above display yields the expression whose existence is claimed in the lemma.
\end{proof}
We continue with a proof of the following inequality.
\begin{lemma}\label{int_pos}
For $g_0$ and $A,B$ as above, the following holds.
\begin{equation}\label{ineq}
-\int _M \Ric (g_0)\wedge (A-B)\geq 0.\\
\end{equation}
\end{lemma}
\begin{proof}
By Chern's result, $-\Ric (g_0)$ represents $2\pi c_1(K)$. Due to the pseudoeffectivity of $K$, we have 
\begin{equation}\label{pos}
- {\sqrt{-1}}^{n-1}\int _M \Ric (g_0) \wedge a_1\wedge \bar a_1  \wedge \ldots \wedge a_{n-1}\wedge \bar a_{n-1}\geq 0
\end{equation}
for smooth $(1,0)$ forms $a_1,\ldots,a_{n-1}$ on $M$.\par
Let $(\chi_\nu)$ be a partition of unity subordinate to the covering $U_\nu$. If $a^{(\nu,k_\nu)}_1\wedge \bar a^{(\nu,k_\nu)}_1  \wedge \ldots \wedge a^{(\nu,k_\nu)}_{n-1}\wedge \bar a^{(\nu,k_\nu)}_{n-1}$ is one of the forms arising in Lemma \ref{lin_alg_claim}, we can turn it into a globally defined form simply by multiplying it with $\chi_\nu$. 
It is immediate that for all $\nu$ and $k_\nu$,
\begin{equation}\label{pos_w_chi}
- {\sqrt{-1}}^{n-1}\int _M \Ric (g_0) \wedge (\chi_\nu a^{(\nu,k_\nu)}_1\wedge \bar a^{(\nu,k_\nu)}_1  \wedge \ldots \wedge a^{(\nu,k_\nu)}_{n-1}\wedge \bar a^{(\nu,k_\nu)}_{n-1})\geq 0,
 \end{equation}
because on $U_\nu$, we can replace $a^{(\nu,k_\nu)}_1$ by $\sqrt{\chi_\nu}a^{(\nu,k_\nu)}_1$ and apply \eqref{pos}.\par
Furthermore, by Lemma \ref{lin_alg_claim},
\begin{gather*} 
-\int _M \Ric (g_0)\wedge (A-B)=\\
-\int _M \Ric (g_0)\wedge ( \sum_\nu {\sqrt{-1}}^{n-1} \chi_\nu\sum_{k_\nu=1}^{N_\nu} a^{(\nu,k_\nu)}_1\wedge \bar a^{(\nu,k_\nu)}_1  \wedge \ldots \wedge a^{(\nu,k_\nu)}_{n-1}\wedge \bar a^{(\nu,k_\nu)}_{n-1})=\\
\sum_\nu  \sum_{k_\nu=1}^{N_\nu}-{\sqrt{-1}}^{n-1}\int _M \Ric (g_0)\wedge (\chi_\nu a^{(\nu,k_\nu)}_1\wedge \bar a^{(\nu,k_\nu)}_1  \wedge \ldots \wedge a^{(\nu,k_\nu)}_{n-1}\wedge \bar a^{(\nu,k_\nu)}_{n-1}).
\end{gather*}
Due to \eqref{pos_w_chi}, the summation in the last line of the above display only contains non-negative numbers, which concludes the proof.
\end{proof}
To finish the proof of Theorem \ref{no_two_signs}, we now obtain a contradiction as follows. Note that the last equality in the display below is due to the fact that $-\Ric (g_-)$ also represents $2\pi c_1(K)$, and that the inequality is due to Lemma \ref{int_pos}.
\begin{align*}
0&=\int _M \Ric (g_0)\wedge A\\
&\leq \int_M \Ric (g_0)\wedge B\\
&= \int_M \Ric (g_-)\wedge B.\\
\end{align*}
The value of the last line in the above display is a positive constant times the total scalar curvature of $g_-$, which yields the desired contradiction.
\section{The case of zero total scalar curvature}
In this section, we conduct the proof of Theorem \ref{mthm2} by assuming that $M$ is not uniruled and then showing that the canonical line bundle of $M$ is torsion. To this end, we will use the following well-known lemma (see \cite[Proposition 3]{Kleiman}, \cite[3.8]{debarre_book}) based on the Hodge Index Theorem.
\begin{lemma}\label{hit_lemma}
Let $X$ be a smooth complex projective variety of dimension $n$. Let $D$ be a divisor on $X$ such that $D\cdot H^{n-1}=0$ for all ample divisors $H$. Then $D$ is numerically trivial.
\end{lemma}
To apply the lemma, take an arbitrary ample divisor on $M$. By Kodaira's embedding theorem, $mH$ is very ample for a sufficiently large positive integer $m$. Let $g$ be the K\"ahler metric that is obtained by restriction of the Fubini-Study metric on projective space after the embedding of $M$ furnished by $mH$. If $\omega$ is the  K\"ahler form associated to $g$, then 
\begin{align*}
~&\ m^{n-1}K\cdot H^{n-1}\\
=&\ -\frac 1 {2\pi}\int _M \Ric (g)\wedge \omega^{n-1}\\
=&\ -\frac 1 {2\pi}\int_M  \frac 2 n s\, \omega^n\\
=&\ 0,
\end{align*}
where the last equality is due to the uniqueness statement in Theorem \ref{no_two_signs} and the assumed existence of {\it some} K\"ahler metric of total scalar curvature equal to zero. By Lemma \ref{hit_lemma}, we can conclude that $K$ is numerically trivial.\par
Moreover, it is a result towards the abundance conjecture due to Kawamata \cite{Kawamata_85} (in the case of minimal varieties) and Nakayama \cite{Nakayama} (in the general case) that in this situation the Kodaira dimension of $M$ is equal to zero, i.e., there exists a positive integer $\ell$ such that $H^0(M,\ell K)$ has a nonzero element $\sigma$. Since $\ell K$ is also numerically trivial, $\divisor (\sigma)=\emptyset $, which implies that $\ell K$ is the trivial line bundle. In other words, $K$ is torsion, q.e.d.\par
We conclude this section by giving an alternative proof of Theorem \ref{mthm2} without using the Hodge Index Theorem in the case when $M$ is a fourfold. It is based on the following special case of the abundance conjecture from \cite[Theorem 9.8]{BDPP_04}.
\begin{theorem}\label{4d_abund}
Let $X$ be a smooth projective fourfold. If $K_X$ is pseudo-effective and if there is a strongly connecting family $(C_t)$ of curves such that $K_X\cdot C_t$=0, then $\kod(X) = 0$.
\end{theorem}
Recall from \cite[p.\ 2]{BDPP_04} that a strongly connecting family $(C_t)$ of curves is defined to be such that any two sufficiently general points can be joined by a chain of irreducible $C_t$'s.\par 
Let $g$ be the K\"ahler metric on $M$ induced by the Fubini-Study metric on the ambient projective space. Again, we denote by $\omega$ the corresponding K\"ahler form. The intersections of $M$ with codimension $3$ planes in projective space yield a strongly connecting family $(C_t)$ of curves due to Bertini's theorem. To determine the intersection number of a curve $C_t$ with $K$, observe that, using again the fact that $-\frac {1}{2\pi}\Ric (g)$ represents $c_1(K)$,
\begin{align*}
K\cdot C_t&=-\frac {1}{2\pi}\int _M \Ric (g)\wedge \omega^3.\\
\end{align*}
However, the value of the right hand side of the above display is, up to a non-zero multiplicative constant, the total scalar curvature of $g$. According to Theorem \ref{no_two_signs}, there cannot exist two K\"ahler metrics with different total scalar curvature on $M$, so the value of the right hand side is zero. Due to Theorem \ref{4d_abund}, $\kod(X) = 0$. To conclude that $K$ is torsion, we can argue as follows (as done by Yau in \cite{Yau_74}).\par
Let $\ell\in \NN^+$ such that there exists $\sigma \in H^0(M,\ell K)\setminus \{0\}$. Assume that $\divisor(\sigma) =D \not = \emptyset$. Then $\Theta=\sqrt{-1}\partial \bar \partial \log |\sigma|^2$ is a closed positive current such that $\Theta$ represents $2\pi\ell c_1(K)$. According to the Poincar\'e-Lelong formula, \label{PL_formula}
\begin{align*} 
0 & < 2\pi \int_D \omega^{n-1} \\
& =  \int_M \Theta\wedge \omega^{n-1}\\
&=-\ell \int _M \Ric (g)\wedge \omega^{n-1}\\
&=-\ell\int_M  \frac 2 n s\, \omega^n \\
& = 0.\\
\end{align*}
Contradiction. Thus, $\divisor(\sigma)=\emptyset$ and $K$ is torsion.

\section{Concluding remarks}\label{sec_concl_rem}
As we remarked in the Introduction, Theorem \ref{mthm1} is almost an equivalence for the following reason. Take an arbitrary K\"ahler metric on a desingularization $\tilde N$ of the variety $N$ in the definition of uniruledness of $M$. Then $\PP^1 \times \tilde N$ carries a product K\"ahler metric of positive (total) scalar curvature simply because the contribution to the curvature from $\tilde N$ can be scaled down until positivity is reached due to the positivity of the Fubini-Study metric on $\PP^1$. However, $\PP^1\times \tilde N$ still rationally dominates $M$. It thus seems likely that $M$ itself carries a K\"ahler metric of positive total scalar curvature, although we know of no rigorous proof.\par
A different way of looking at the possibility of a converse of Theorem \ref{mthm1} is via \cite{BDPP_04}. Recall that their main technical result is the following, which, combined with \cite{MM}, immediately implies the above-used \cite[Corollary 0.3]{BDPP_04}.
\begin{theorem}[{\cite[Theorem 2.2]{BDPP_04}}]\label{bdpp_mthm} 
Let $X$ be a complex projective manifold. Then a class $\alpha\in {\rm NS}_\RR(X)$ is pseudo-effective if (and only if) it is in the dual cone of the cone ${\rm SME}(X)$ of strongly movable curves.
\end{theorem}
Now, if $M$ is uniruled, and thus $K$ is not pseudo-effective, Theorem \ref{bdpp_mthm} yields, by the definition of strongly movable curves \cite[Definition 1.3(v)]{BDPP_04}, that there exist a modification $\mu:\tilde M \to M$ and very ample divisor classes $\tilde A_1,\ldots,\tilde A_{n-1}$ on $\tilde M$ such that
$$\int_M c_1(M)\wedge \mu_*(\tilde A_1\wedge \ldots\wedge \tilde A_{n-1}) >0.$$ 
It thus seems probable that a uniruled $M$ itself possesses a K\"ahler metric of positive total scalar curvature, although we again do not have a complete proof. For some further evidence concerning this, we refer to \cite[Corollary 5.18]{Hitchin} and \cite[Proposition 1]{Yau_74}.\par
Moreover, we remark that Theorem \ref{mthm2} {\it is} an equivalence except in the uniruled case. Namely, if the canonical line bundle of $M$ is torsion, then the first real Chern class of $M$ is equal to zero. By Yau's solution of the Calabi Conjecture \cite{Yau}, there is a Ricci-flat K\"ahler metric on $M$, which has zero (total) scalar curvature.\par
For the future, we think it might be interesting to explore implications of other forms of positive curvature in terms of the existence of rational curves. This is motivated to a significant extent by the pioneering work of Yau \cite{Yau_74}. We hope to return to this subject in a subsequent paper.\par
Finally, we address a suggestion from one of the referees by introducing the cohomological invariant
$$T(\omega)=\int _M c_1(M)\wedge \omega ^{n-1},$$
where $\omega$ is an arbitrary K\"ahler form on $M$. Note that $T(\omega)$ simply equals $\frac{(n-1)!}{\pi}$ times the total scalar curvature of the associated K\"ahler metric $g$. With the help of $T$, our results can be summarized as follows.
\begin{itemize}
\item[(1)] If $T$ assumes a positive value somewhere, then $M$ is uniruled.\\
\item[(2)] If $T$ assumes more than one sign from $\{+,0,-\}$, then $M$ is uniruled.\\
\item[(3)] If $M$ is not uniruled and $T$ assumes the value zero somewhere, then $T$ is identically equal to zero and $M$ is Calabi-Yau.
\end{itemize}


\begin{thebibliography}{BDPP04}

\bibitem[BDPP04]{BDPP_04}
S.~Boucksom, {J.-P.}~Demailly, M.~Paun, and P.~Peternell.
\newblock The pseudo-effective cone of a compact {K}\"ahler manifold and
  varieties of negative {K}odaira dimension. Current version available electronically at {\tt www-fourier.ujf-grenoble.fr/$_{\widetilde{~}}$demailly/research.html}. Earlier version: {\em {\rm arXiv:math/0405285v1}}, 2004.

\bibitem[Ber66]{Berger}
M.~Berger.
\newblock Sur les vari\'et\'es d'{E}instein compactes.
\newblock In {\em Comptes {R}endus de la {III}e {R}\'eunion du {G}roupement des
  {M}ath\'ematiciens d'{E}xpression {L}atine ({N}amur, 1965)}, pages 35--55.
  Librairie Universitaire, Louvain, 1966.

\bibitem[Deb01]{debarre_book}
O.~Debarre.
\newblock {\em Higher-dimensional algebraic geometry}.
\newblock Universitext. Springer-Verlag, New York, 2001.

\bibitem[Dem01]{Dem}
J.-P.~Demailly.
\newblock Multiplier ideal sheaves and analytic methods in algebraic geometry.
\newblock In {\em School on {V}anishing {T}heorems and {E}ffective {R}esults in
  {A}lgebraic {G}eometry ({T}rieste, 2000)}, volume~6 of {\em ICTP Lect.
  Notes}, pages 1--148. Abdus Salam Int. Cent. Theoret. Phys., Trieste, 2001.

\bibitem[Hit75]{Hitchin}
N.~Hitchin.
\newblock On the curvature of rational surfaces.
\newblock In {\em Differential geometry ({P}roc. {S}ympos. {P}ure {M}ath.,
  {V}ol. {XXVII}, {P}art 2, {S}tanford {U}niv., {S}tanford, {C}alif., 1973)},
  pages 65--80. Amer. Math. Soc., Providence, R. I., 1975.

\bibitem[HLW10]{heier_lu_wong}
G.~Heier, S.~S.~Y. Lu, and B.~Wong.
\newblock On the canonical line bundle and negative holomorphic sectional
  curvature.
\newblock {\em Math. Res. Lett.}, 17(6):1101--1110, 2010.

\bibitem[Kaw85]{Kawamata_85}
Y.~Kawamata.
\newblock Minimal models and the {K}odaira dimension of algebraic fiber spaces.
\newblock {\em J. Reine Angew. Math.}, 363:1--46, 1985.

\bibitem[Kle66]{Kleiman}
S.~Kleiman.
\newblock Toward a numerical theory of ampleness.
\newblock {\em Ann. of Math. (2)}, 84:293--344, 1966.

\bibitem[Kol96]{kollar_book}
J.~Koll{\'a}r.
\newblock {\em Rational curves on algebraic varieties}, volume~32 of {\em
  Results in Mathematics and Related Areas. 3rd Series. A Series of Modern
  Surveys in Mathematics}.
\newblock Springer-Verlag, Berlin, 1996.

\bibitem[Miy87]{M1}
Y.~Miyaoka.
\newblock Deformations of a morphism along a foliation and applications.
\newblock In {\em Algebraic geometry, {B}owdoin, 1985 ({B}runswick, {M}aine,
  1985)}, volume~46 of {\em Proc. Sympos. Pure Math.}, pages 245--268. Amer.
  Math. Soc., Providence, RI, 1987.

\bibitem[MM86]{MM}
Y.~Miyaoka and S.~Mori.
\newblock A numerical criterion for uniruledness.
\newblock {\em Ann. of Math. (2)}, 124(1):65--69, 1986.

\bibitem[Mor79]{Mori_79}
S.~Mori.
\newblock Projective manifolds with ample tangent bundles.
\newblock {\em Ann. of Math. (2)}, 110(3):593--606, 1979.

\bibitem[Nak04]{Nakayama}
N.~Nakayama.
\newblock {\em Zariski-decomposition and abundance}, volume~14 of {\em MSJ
  Memoirs}.
\newblock Mathematical Society of Japan, Tokyo, 2004.

\bibitem[Siu10]{siu}
Y.~T. Siu.
\newblock Abundance conjecture.
\newblock {\em {\rm arXiv:math.AG/0912.0576v3}}, 2010.

\bibitem[SY80]{siu_yau}
Y.~T. Siu and S.~T. Yau.
\newblock Compact {K}\"ahler manifolds of positive bisectional curvature.
\newblock {\em Invent. Math.}, 59(2):189--204, 1980.

\bibitem[Yau74]{Yau_74}
S.~T. Yau.
\newblock On the curvature of compact {H}ermitian manifolds.
\newblock {\em Invent. Math.}, 25:213--239, 1974.

\bibitem[Yau78]{Yau}
S.~T. Yau.
\newblock On the {R}icci curvature of a compact {K}\"ahler manifold and the
  complex {M}onge-{A}mp\`ere equation. {I}.
\newblock {\em Comm. Pure Appl. Math.}, 31(3):339--411, 1978.

\end{thebibliography}
\end{document}